\documentclass[12pt]{amsart}
\usepackage{amssymb}
\usepackage{geometry}
\usepackage{amsmath}
\usepackage{amsfonts}
\usepackage[mathscr]{eucal}
\usepackage{amsthm}
\usepackage{bookmark}
\usepackage{enumerate}
\usepackage{extarrows}

\usepackage{changes}%\usepackage[final]{changes}  %\listofchanges
\usepackage{color}
\usepackage{bbm}
\usepackage{verbatim}
\usepackage{graphicx}
\usepackage{courier}
\usepackage{helvet}         % selects\textbf{\textbf{a?¡é}} Helvetica as sans-serif font
\usepackage{courier}        % selects Courier as typewriter font
\usepackage{type1cm}        % activate if the above 3 fonts are
\usepackage{multicol}        % used for the two-column index
\usepackage[bottom]{footmisc}

\date{}

\theoremstyle{plain}

\newtheorem{theorem}{Theorem}[section]
\newtheorem{proposition}[theorem]{Proposition}
\newtheorem{lemma}[theorem]{Lemma}
\newtheorem{corollary}[theorem]{Corollary}

\theoremstyle{definition}

\numberwithin{equation}{section} \theoremstyle{plain}

\author%[authorlabel1]
{Shilei Fan}
\address%[authorlabel1]
{Shilei FAN: School of Mathematics and Statistics, and Key Lab NAA--MOE, Central China Normal University, Wuhan 430079, China}
\email{slfan@ccnu.edu.cn}

\title{ Tiling the field $\Q_p$ of $p$-adic numbers by a function }
\date{\today}

\def\R{{\textbf{R}}}

\def\R{\mathbb R}

\def\1{\mathbbm 1}
\def\Z{\mathbb Z}
\def\Q{\mathbb Q}

\def\m{\mathfrak{m}}
\thanks{ S. L. FAN was partially supported by NSFC (grants No. 12331004  and No.  12231013). }
\begin{document}
	\maketitle
	\begin{abstract}

This study explores  the  properties  of the function which can tile  the field $\Q_p$ of $p$-adic numbers by translation. It is  established that functions capable of tiling $\Q_p$ is by translation uniformly locally constancy.
As an application, in the field $\Q_p$, we addressed the question posed by H. Leptin and D. M\"uller, providing the necessary and sufficient conditions for a discrete set to correspond to a uniform partition of unity. The study also connects these tiling properties to the Fuglede conjecture, which states that a measurable set is a tile if and only if it is spectral. The paper concludes by characterizing the structure of tiles in \(\mathbb{Q}_p \times \mathbb{Z}/2\mathbb{Z}\), proving that they are spectral sets.

	\end{abstract}
	\section{Introduction}
Consider a locally compact abelian group $G$ equipped with a Haar measure $\mu$. Let $f$ be a function in $L^1(G)$.
The function $f$ is said to {\em tile} $G$ at level $w$ by a translation set $T$ if
\[ \sum_{t \in T} f(x - t) = w \quad \text{a.e.} \ x\in G. \]
Here, the convergence is absolute,   the translation set $T$ is required to be locally finite, i.e. for any compact set $K$, the set $T\cap K$ is finite and we allow elements of $T$ to occur with finite multiplicities. We say that $(f, T)$ is a tiling pair at level $w$. If $f = 1_\Omega$ is an indicator function of the measurable set $\Omega$ and $w = 1$, we say that $\Omega$ {\em tiles} the group $G$ by the translation set $T$, and $(\Omega, T)$ is a {\em tiling pair} (at level $1$). The set $\Omega$ is called a {\em tile } of $G$

In the case where \( G = \mathbb{R} \), it was proven in \cite{LW96Invent} that if a bounded measurable set \(\Omega\) with measure-zero boundary tiles the line \(\mathbb{R}\) by a translation set \(T\), then \( T \) is a periodic set, that is, \( T + \tau = T \) for some \( \tau > 0 \). This result was extended in \cite{KL96Duke} (and proved earlier in \cite{LM91}) to tiling \(\mathbb{R}\) by a function \( f \in L^1(\mathbb{R}) \) with compact support. It was shown that if \( (f, T) \) is a tiling pair at some level \( w \) with \( T \) of bounded density and \( f \) is not identically zero, then \( T \) must be a finite union of periodic sets. For more related results and questions, interested readers can refer to \cite{KL21}.

For the case \( G = \mathbb{Q}_p \), the field of \( p \)-adic numbers, it was proven in \cite{FFLS} that if \((\Omega, T)\) forms a tiling pair, then \( \Omega \) is a compact open set up to a set of measure zero, which implies
\[
1_\Omega(x) = 1_{\Omega}(x + \tau), \quad \text{a.e.} \ x \in \mathbb{Q}_p
\]
for some \( \tau \in \mathbb{Q}_p \setminus \{0\} \).
This means that \( 1_\Omega(x) \) is a periodic function or the measurable set \( \Omega \) is periodic.

%For the case \( G = \mathbb{Q}_p \),the field of \( p \)-adic numbers, it was proven in \cite{FFLS} that if   $(\Omega,T)$  forms a tiling pair, then \( \Omega \) is a compact open set up to a set of measure zero, which implies
%\[
%1_\Omega(x) = 1_{\Omega}(x + \tau), \quad  \text{a.e.} \ x\in \mathbb{Q}_p
%\]
% for some $ \tau \in \mathbb{Q}_p \setminus \{0\}$.
%This means that \( 1_\Omega(x) \) is a periodic function or the measurable set \( \Omega \) is periodic.

Let \( \mathbb{Z}_p \) be the ring of \( p \)-adic integers. A compact open set \( \Omega \) can be expressed as
\[
\Omega = \bigsqcup_{c \in C} (c + p^\gamma \mathbb{Z}_p)
\]
for some finite set \( C \subset \mathbb{Q}_p \) and some integer \( \gamma \in \mathbb{Z} \). Note that \( p^\gamma \mathbb{Z}_p \) is a subgroup of \( \mathbb{Q}_p \) under addition. Hence, a compact open set is a finite union of cosets of some subgroup. Moreover, the structure of the tile \( \Omega \) and the translation set \( T \) are characterized by \( p \)-homogeneous trees in \cite{FFS}.

Let $f \in L^{1}(\mathbb{Q}_p)$ be a function. We say that $f$ is \emph{periodic} if
\[f(x + \tau) = f(x) \quad \text{a.e.  } x \in \mathbb{Q}_p,\]
 for some $\tau \in \mathbb{Q}_p \setminus \{0\}$. In this note, we show that any function capable of tiling $\mathbb{Q}_p$ is inherently periodic.
 
 \begin{theorem}\label{thm-tilingperiodic}
 Let $f \in L^{1}(\mathbb{Q}_p)$ and $V \subset \mathbb{Z} \setminus \{0\}$ be a finite set of non-zero integers. Suppose that $T$ is a locally finite  subset in $\mathbb{Q}_p$, and $v_t \in V$ for $t \in T$ are such that
\begin{align}
\sum_{t \in T} v_t \cdot f(x-t) = w, \quad \text{a.e. } x \in \mathbb{Q}_p,
\end{align}
for some $w \in \mathbb{R}$.Then $f$ is  uniformly locally constancy, i.e. there exists $n\in \mathbb{Z}$ such that
\begin{align}\label{eq-periodic}
\forall u\in B(x,p^n), \quad f(x+u)=f(x) \quad  \text{a.e. } x\in \Q_p.
\end{align}
 \end{theorem}
Remark that Equality \eqref{eq-periodic} implies that $f$ is periodic.
  Actually, the definition regarding tiling doesn't necessarily require the group to be commutative (abelian). It is proved in \cite{HN79} that in every connected nilpotent group $G$ there exists a discrete subset $T$ and a corresponding non-negative smooth function $\phi$ with compact support in $G$ such that
\[\sum_{t\in T} \phi(t x)=1 \quad \text{for all } x\in G,\]
i.e., the family $\{\phi(t x): t\in T\}$ forms a partition of unity in $G$, which is called a \emph{uniform partition of unity} in $G$. H. Leptin and D. M\"uller \cite{LM91} proposed two questions:

1. Which locally compact groups do admit uniform partitions of unity?

2. Can one describe the set $T$ which corresponds to a uniform partition of unity?\\
And a positive answer for the first question was obtained by the authors. However, the authors claimed that a comprehensive answer to the second question seems to be beyond the present possibilities. They restricted their considerations to the simplest non-trivial case, namely to the group of reals: $ G = \mathbb{R}$ and provided a complete solution.

In the case of $G=\mathbb{Q}_p$, the field of $p$-adic numbers, we have also presented a complete solution. For $x\in \Q_p$ and $r>0$, denote by \[B(x, r)=\{y\in \Q_p: |x-y|_p\leq r\}\] the closed  ball of radius $r$ centered at $x$.

\begin{theorem} \label{thm:card}
A discrete set $T \subset \mathbb{Q}_p$ corresponds to a uniform partition of unity if and only if there exists a sufficiently large integer $n$ such that \[ \#(B(x,p^n)\cap T) =\#(B(y,p^n)\cap T) \] for all $x, y \in \mathbb{Q}_p$.

\end{theorem}

 For other motivation, it is illuminating to observe the inherent connection between tiling by a set and tiling by a function, particularly within the framework of the Fuglede conjecture \cite{MR0470754}. The conjecture stats that a measurable set $\Omega \in \R^d$ of unit measure  can tile $\R^d$ by translation if and only if there exists a set $\Lambda \in \R^d$ such that $\{e^{2\pi i \lambda x}, \lambda \in \Lambda\}$ forms an orthonormal basis of $L^2(\Omega)$. Here, the set $\Lambda$ is called a spectrum of $\Omega$ and $(\Omega,\Lambda)$ is called a spectral pair. Denote by $f_\Omega$ the so-called power spectrum of $1_\Omega$
 \[f_{\Omega}(\xi)=|\widehat{1_{\Omega}}(\xi)|^2.\]
 Then $(\Omega,\Lambda)$ is a spectral pair if and only if $(f_{\Omega},\Lambda)$ form a tiling pair at level $1$.
Hence, the Fuglede conjecture takes the following symmetric form
\begin{align}\label{equ-fug}
 \Omega\text{ tiles }\R^d \Longleftrightarrow |\widehat{1_{\Omega}}|^2 \text{ tiles } \R^d.
\end{align}

 There are many positive results under different extra assumptions before the work\cite{MR2067470} where Tao gave a counterexample: there exists a spectral subset of ${\R}^{n}$ with $n\ge 5$ which is not a tile. After that, Matolcsi\cite{MR2159781}, Matolcsi and Kolountzakis\cite{MR2264214,MR2237932}, Farkas and Gy \cite{MR2221543}, Farkas, Matolcsi and Mora\cite{MR2267631} gave a series of counterexamples which show that both directions of Fuglede's conjecture fail in ${\R}^{n}\left ( n\ge 3 \right ) $. However, the conjecture is still open in low dimensions $n=1,2$.

 Generally, for a locally compact abelian group $G$, denote by  $\widehat{G} $ its dual group.  Consider a Borel measurable subset $\Omega $ in $G$  of unit measure. We say that $\Omega $ is a \textit{spectral set} if there exists a set $\Lambda \subset \widehat{ G} $ which forms an orthonormal basis of the Hilbert space $L^{2} \left ( \Omega  \right ) $. Such a set $\Lambda $ is called a \textit{spectrum} of $\Omega $ and $\left ( \Omega ,\Lambda  \right ) $ is called a \textit{spectral pair}. The Fuglede conjecture can be generalized to locally compact abelian groups: {\em a Borel measurable set $\Omega \subset G$ of unit measure is a spectral set if and only if it can tile $G$ by translation.} Similarly, the generalized  Fuglede conjecture also takes the following symmetric form
\begin{align}\label{equ-fugloc}
 \Omega\text{ tiles }G \Longleftrightarrow |\widehat{1_{\Omega}}|^2 \text{ tiles } G.
\end{align}

In its generality, this generalized conjecture does not hold true. Instead, we should inquire about the groups for which it holds. This question even arises for finite groups. The counterexamples in $\mathbb{R}^d$, $d\geq 3$, are actually constructed based on counterexamples in finite groups. Substantial work has been done for some finite groups \cite{MR3684890, MR4402630, FFS, MR3649367, MR4085122, MR4493731, MR1772427, MR2237932, MR1895739, MR4422438, MR3695475, MR4186120, MR4604197}.
For infinite groups, based on \eqref{equ-fugloc}, Fan et al. \cite{FFLS} proved that Fuglede's conjecture holds in the field $\mathbb{Q}_p$ of $p$-adic numbers. In fact, this is the first example of an infinite abelian group where Fuglede's conjecture holds. It is natural to find other infinite locally compact abelian groups where Fuglede's conjecture holds.

As an application of Theorem \ref{thm-tilingperiodic},  we can characterize the structure of the tiles in  the infinite group $\Q_p\times \Z/2\Z$.

%
%\begin{theorem}\label{thm1.3}
%Suppose  $p>2$ is  a prime number,  and   $\Omega= \Omega_0\times\{0\} \cup \Omega_1\times\{1\} $ tiles $\Q_p\times \Z/2\Z$ by translation. 
%
%Then  there are two mutually exclusive cases:
%\begin{enumerate}[{\rm (1)}]
%\item $\Omega_0$ and $\Omega_1$  can tiles $\Q_p$ with  a common translation set  $T_0\subset \Q_p$,
%\item $\mu(\Omega_0\cap\Omega_1)=0$ and  $ \Omega_0\cup \Omega_1$ can tile $\Q_p$ by translation.
%\end{enumerate}
%\end{theorem}

\begin{theorem}\label{thm1.3}
Assume that  $\Omega= \Omega_0\times\{0\} \cup \Omega_1\times\{1\} $ tiles $\Q_p\times \Z/2\Z$ by translation. 
\begin{enumerate}[{\rm (1)}]
\item If $p>2$, then either $\mu(\Omega_0\cap\Omega_1)=0$ and  $ \Omega_0\cup \Omega_1$ can tile $\Q_p$ by translation   or  $\Omega_0$ and $\Omega_1$  can tiles $\Q_p$ with  a common translation set  $T_0\subset \Q_p$. 
\item If $p=2$, then  we have three cases: 
\begin{itemize}
\item[(i)] $\mu(\Omega_0\cap \Omega_1)=0$ and $\Omega_0\cup \Omega_1$ tiles $\Q_2$ by translation,

\item[(ii)]$\Omega_0$ and $\Omega_1$  can tiles $\Q_2$ with  a common translation set  $T_0\subset \Q_2$
 
\item[(iii)] both $\Omega_0$ and $\Omega_1$ are compact open (up to a measure zero set)   and  there exist $j_0\in \{0,\cdots, n-1\}$  such that
$\Omega_0$ and  \[\widetilde{\Omega}_1 = \{ x + 2^{n-j_0-1}: x  \in \Omega_1\}\] are disjoint (up to a measure zero set) and  $\Omega_0\cup \widetilde{\Omega}_1$ tiles  $\Q_2$ by translation.
\end{itemize}
\end{enumerate}

%Then  there are two mutually exclusive cases:
%\begin{enumerate}[{\rm (1)}]
%\item $\Omega_0$ and $\Omega_1$  can tiles $\Q_p$ with  a common translation set  $T_0\subset \Q_p$,
%\item $\mu(\Omega_0\cap\Omega_1)=0$ and  $ \Omega_0\cup \Omega_1$ can tile $\Q_p$ by translation.
%\end{enumerate}
\end{theorem}

As a consequence, we establish that  spectrality of tiles in the infinite group $\Q_p\times \Z/2\Z$.

\begin{corollary} Tiles in  $\Q_p \times \Z/2\Z$ are spectral sets.
\end{corollary}

\section{Preliminaries}

\subsection{The field  $\Q_p$ of $p$-adic numbers }\label{p-adicfield}
Let's start with a quick review of \(p\)-adic numbers. Consider the field \(\mathbb{Q}\) of rational numbers and a prime \(p \ge 2\). Any nonzero number \(r \in \mathbb{Q}\) can be expressed as \(r = p^v \frac{a}{b}\), where \(v, a, b \in \mathbb{Z}\) and \((p, a) = 1\) and \((p, b) = 1\) (here \((x, y)\) denotes the greatest common divisor of the integers \(x\) and \(y\)).
We define \(|r|_p = p^{-v_p(r)}\) for \(r \neq 0\) and \(|0|_p = 0\). Then \(|\cdot|_p\) is a non-Archimedean absolute value, meaning:
\begin{itemize}
    \item[(i)] \(|r|_p \ge 0\) and equality only when \(r = 0\),
    \item[(ii)] \(|r s|_p = |r|_p |s|_p\),
    \item[(iii)] \(|r + s|_p \leq \max\{ |r|_p, |s|_p\}\).
\end{itemize}

The field \(\mathbb{Q}_p\) of \(p\)-adic numbers is the completion of \(\mathbb{Q}\) under \(|\cdot|_p\). The ring \(\mathbb{Z}_p\) of \(p\)-adic integers is the set of \(p\)-adic numbers with absolute value at most 1. A typical element \(x\) of \(\mathbb{Q}_p\) is written as
\begin{equation}\label{HenselExp}
    x = \sum_{n = v}^\infty a_n p^{n} \quad (v \in \mathbb{Z}, a_n \in \{0,1,\dotsc, p-1\}, \text{ and } a_v \neq 0).
\end{equation}
Here, \(v_p(x) := v\) is called the \(p\)-{\em valuation} of \(x\).

A non-trivial continuous additive character on $\mathbb{Q}_p$ is defined by
$$
\chi(x) = e^{2\pi i \{x\}},
$$
where $\{x\} = \sum_{n=v_p(x)}^{-1} a_n p^n$ is the fractional part of $x$ in (\ref{HenselExp}).
Remark that
\begin{align}\label{one-in-unit-ball}
\chi(x) = e^{2\pi i k/p^n}, \quad \text{if } x \in \frac{k}{p^n} + \mathbb{Z}_p \ \  (k, n \in \mathbb{Z}),
\end{align}
and
\begin{align}\label{integral-chi}
\int_{p^{-n}\mathbb{Z}_p} \chi(x) \,dx = 0 \ \text{for all } n \geq 1.
\end{align}

 From the character $\xi$, we can obtain all characters $\chi_\xi$ of $\mathbb{Q}_p$ by defining
$\chi_\xi(x) = \chi(\xi x)$.
The map $\xi \mapsto \chi_{\xi}$ from $\mathbb{Q}_p$ to $\widehat{\mathbb{Q}}_p$ is an isomorphism. We write $\widehat{\mathbb{Q}}_p \simeq \mathbb{Q}_p$ and identify a point $\xi \in \mathbb{Q}_p$ with the point $\chi_\xi \in \widehat{\mathbb{Q}}_p$.
For more information on $\mathbb{Q}_p$ and $\widehat{\mathbb{Q}}_p$, the reader is referred to the book \cite{VVZ94}.

The following notation will be used in the whole paper.
\medskip
\\
\noindent {\bf Notation}:
\begin{itemize}
\item $\mathbb{Z}_p^\times := \mathbb{Z}_p\setminus p\mathbb{Z}_p=\{x\in \mathbb{Q}_p: |x|_p=1\}$,
the group of units of $\mathbb{Z}_p$.

\item $B(0, p^{n}): = p^{-n} \mathbb{Z}_p$,  the (closed) ball centered at $0$ of radius $p^n$.

\item $B(x, p^{n}): = x + B(0, p^{n})$.

\item $ S(x, p^{n}): = B(x, p^{n})\setminus B(x, p^{n-1})$,  a ``sphere".

%\item $ \mathbb{L} : = \{\{x\}: x\in \Q_p\}$, a complete set of representatives of the cosets of the additive subgroup $\mathbb{Z}_p$.
%
%\item $ \mathbb{L}_n : = p^{-n} \mathbb{L}$.
\end{itemize}
%We also use it to denote balls in $\Q_p$.

\subsection{Fourier Transform}
The Fourier transform of  $f\in L^1(\Q_p)$ is defined to be
$$\widehat{f}(\xi)=\int_{\Q_p}f(x)\overline{\chi_\xi(x)} dx   \quad (\forall \xi\in \widehat{\Q}_p\simeq \Q_p).$$

A complex function $f$ defined on $\Q_p$ is called \textit{uniformly locally constancy} if  there exists $n\in \mathbb{Z}$ such that
\[f(x+u)=f(x) \quad \forall x\in \Q_p, \forall u \in B(0, p^n).\]
The following proposition shows that for  an integrable function $f$, having compact support and being uniformly locally constant are dual properties for $f$ and its Fourier transform.

\begin{proposition}[\cite{FFLS}]\label{Prop-compactConstant}
	Let $f\in L^1(\Q_p)$ be  a complex-value integrable function. Then
  $f$ has compact support if and only if  $\widehat{f}$ is uniformly locally constancy.
\end{proposition}
The Fourier transform of   a  positive finite measure $\mu $ on $\Q_p$ is defined to be
\[\widehat{\mu}(\xi)=\int_{\Q_p}\overline{\chi_\xi(x)} d\mu(x)   \quad (\forall \xi\in \widehat{\Q}_p\simeq \Q_p).\]
%Let $\mu $ be a real measure on a $\sigma $-algebra $\mathfrak{M} $, denote by  $\left | \mu  \right | $
%	 the total variation of $\mu $, define
%	$$\mu ^{+} =\frac{1}{2} \left ( \left | \mu  \right | +\mu  \right ) ,\mu ^{-} =\frac{1}{2} \left ( \left | \mu  \right | -\mu  \right )$$
%		The measures $\mu ^{+}$ and $\mu ^{-}$ are called the positive and negative variations of $\mu$.
Let $\nu $ be a signed measure on ${\Q}_{p}$,  with
	 the finite total variation $\left | \nu  \right | $.  %define  $$\mu ^{+} =\frac{1}{2} \left ( \left | \mu  \right | +\mu  \right ) ,\mu ^{-} =\frac{1}{2} \left ( \left | \mu  \right | -\mu  \right )$$
		The measures
		$$\nu ^{+} =\frac{1}{2} \left ( \left | \nu  \right | +\nu  \right ) ,\mu ^{-} =\frac{1}{2} \left ( \left | \nu  \right | -\nu  \right )$$ are called the positive and negative variations of $\nu$.
	The Fourier transform of  $\nu $ is defined to be  
\[\widehat{\nu}(\xi) =\widehat{\nu^{+} }(\xi)-\widehat{\nu^{-} }(\xi)=\int_{{\Q}_{p} }\overline{\chi_{\xi} (x) } d\nu^{+}(\xi) -\int_{{\Q}_{p} }\overline{\chi_{\xi} (x) } d\nu^{-}(x)=\int_{{\Q}_{p} }\overline{\chi_{\xi} (x) } d\nu(x)\]
%-\int_{{\Q}_{p} }\overline{ \chi_{\xi} ( x \right )} d\nu ^{-} (x)= \int_{{\Q}_{p} }\overline{ \chi_{\xi} ( x \right )} d\nu ^{-} (x) \]
	for all $\xi \in \widehat{{\Q}}_{p}\cong {\Q}_{p}$.

\subsection{$\mathbb{Z}$-module generated by $p^n$-th roots of unity} \label{Zmodule}

Let $m\ge 2$ be an integer and let  $\omega_m = e^{2\pi i/m}$, which is a primitive $m$-th root of unity. Denote  by $\mathcal{M}_m$  the set of integral points
$(a_0, a_1, \dotsc, a_{m-1}) \in \mathbb{Z}^m$ such that
$$
\sum_{j=0} ^{m-1} a_j \omega_m^j =0.
$$
The set $\mathcal{M}_m$ is clearly a $\mathbb{Z}$-module. When $m=p^n$ is a power of a prime number, the  relation between the  coefficients is established.
\begin{lemma} [{\cite[Theorem 1]{Sch64}}]\label{SchLemma}
If $(a_0,a_1,\dotsc, a_{p^n-1})\in  \mathcal{M}_{p^n}$,
	then for any integer $0\le j\le p^{n-1}-1$ we have $a_j=a_{j+sp^{n-1}}$ for all $s=0,1,\dots, p-1$.
\end{lemma}

A  subset $C\subset \Q_p$ is called a  {\em $p$-cycle}  in $ \Q_{p}$
is a set of $p$  elements  such that
\[\sum_{c\in C} \chi(c)=0\]
which is equivalent to
\begin{align}\label{eq-pcycle}
C \equiv \left \{ r,r+\frac{1}{p} ,r+\frac{2}{p} ,\dots ,r+\frac{p-1}{p}  \right \} \mod \Z_{p}
\end{align}
for some $r\in \Q_{p} $. Note that  \eqref{eq-pcycle} is also equivalent to that $C$ is a subset of $p$ elements such that for distinct $c, c^{\prime}\in C$,
\[ |c-c^{\prime}|_p=p.\]
For a $p$-cycle $C$ in $\Q_p$ and $x\in \Z_p^{\times}$, it is easy to see that $x\cdot C$ is also a $p$-cycle in $\Q_p$, which is to say
\[\sum_{c\in C} \chi(x c)=0.\]

Let $ E\subset \Q_{p}$ be a finite subset.  If $\sum_{c\in E} \chi \left ( c  \right ) =0$,  then by Lemma \ref{SchLemma}, $E$ can be decomposed into some $p$-cycles.
Hence, the property $\sum_{c\in E} \chi(c) = 0$ of the set $E\subset \Q_p$ is invariant under rotations.
\begin{lemma}[{\cite[ Lemma 2.6]{FFS}}]\label{lem6}
	Let $\xi _{0}, \xi _{1},\cdots ,\xi _{m-1}$ be $m$ points in $\Q_{p}$. If\: $\sum_{j=0}^{m-1} \chi \left ( \xi _{j}  \right ) =0$, then $p\mid m$ and $$\sum_{j=0}^{m-1} \chi \left ( x\xi _{j}  \right ) =0$$for all $x\in \Z_{p}^{\times } $.
\end{lemma}
Actually, by  Lemmas  \ref{SchLemma} and \ref{lem6}, we have the following corollary.
\begin{corollary}\label{corollary2}
	Let $\xi _{0}, \xi _{1},\cdots ,\xi _{m-1}$ be $m$ points in $\Q_{p}$. If $\sum_{j=0}^{m-1} \alpha _{j} \chi \left ( \xi _{j}  \right ) =0$, where $\alpha _{j}\in \Z $, then $$\sum_{j=0}^{m-1}\alpha _{j} \chi \left ( x\xi _{j}  \right ) =0$$for all $x\in \Z_{p}^{\times } $.
\end{corollary}
%\begin{proof}
%\end{proof}

\subsection{Zeros of the Fourier Transform of a finite discrete signed measure}
Let   $E$ be a  non-empty finite subset of $\Q_p$ and  \[\nu=\sum_{x\in E} \alpha_{x}\delta_x,\quad  \alpha_{x}\in \Z\setminus\{0\}\]
be a  finite  measure supported in  the set $E$.
By definition, we have
\[ \widehat{\nu} (\xi)= \sum_{x\in E} \alpha_{x}\overline{\chi(\xi x)}= \sum_{x\in E} \alpha_{x}e^
{2\pi  i\{-\xi  x\}}.\]

By Corollary \ref{corollary2}, the zero set of the Fourier transform of a finite discrete measure $\nu$ is invariant under rotations.

\begin{lemma}\label{lem-fourierzero}
If $\widehat{\nu}(\xi)=0$, then  $\widehat{\nu}(\xi \cdot \zeta)=0$ for all $\zeta\in  \Z_{p}^{\times }$.
\end{lemma}
A necessary condition for a point to be a zero of the Fourier transform of $\nu$ is provided.
\begin{lemma}\label{lem-necessary}
For $\xi\in \Q_p^{\times}$, if $\widehat{\nu}(\xi)=0$, then  for any $x\in E$, there exists  a distict point  $x^{\prime}\in E$ such that
\[|x-x^{\prime}|_p\leq  \frac{p}{|\xi|_p}. \]
\end{lemma}
\begin{proof}
Since $E$ is a finite subset of $\mathbb{Q}_p$, it is bounded. Specifically, let's assume that for each $x \in E$, the $p$-adic absolute value $|x|_p$ satisfies $|x|_p \leq p^n$ for some integer $n$.

For $\xi \in \Q_p^{\times}$, we have $|\xi|_p=p^m$ for some $m\in \Z$. By Lemma \ref{lem-fourierzero}, $\widehat{\nu} (\xi)=0$ if and only if $\widehat{\nu} (-1/p^m)=0$, which is equivalent to
\begin{align}\label{eq-Fourierzero}
  \sum_{x\in E} \alpha_{x}\chi_{1/p^m}(x)=0.
\end{align}

Note that the character $\chi_{1/p^m}$ is uniformly locally constant on the disks with radius $1/p^m$. Consider the map $\chi_{1/p^m}:E \to \{1, e^{2\pi i /p^{m+n}}, \cdots, e^{2\pi i (p^{m+n}-1)/p^{m+n}}\}$. Denote by $E_j=\{x\in E: \chi_{1/p^m}(x)=e^{2\pi i j/p^{m+n}} \}$.

Equality \eqref{eq-Fourierzero} is equivalent to
\[
\sum_{j=0} ^{p^{m+n}-1} a_j  e^{2\pi i j /p^{m+n}} =0,
\]
where $a_j=\sum_{x\in E_j} \alpha_x$.

For $x\in E$, assume that $\chi_{1/p^m}(x)=e^{2\pi i j /p^{m+n}} $ for some $0\leq j\leq p^{m+n}-1$. By Lemma \ref{SchLemma}, $a_j=a_{j+p^{m+n-1} \mod p^{m+n}}$. If $a_j=0$, then there exists $x^{\prime}\in E_j$ such that $x^{\prime}\neq x$. Noting that $\chi_{1/p^m}(x)= \chi_{1/p^m}(y)$ if and only if $|x-y|_p\leq 1/p^m$. Hence $|x-x^{\prime}|\leq 1/p^m$. On the other hand, if $a_j\neq 0$, then there exists $x^{\prime}\in E$ such that
\[\chi_{1/p^m}(x^{\prime})=e^{2\pi i \frac{j+p^{m+n-1}}{p^{m+n}} },\]
which implies \[|x-x^{\prime}|_p=\frac{1}{p^{m-1}}.\]

%Since $E$ is a finite subset of $\mathbb{Q}_p$, it follows that $E$ is bounded. Specifically, assume that for each $x \in E$, the $p$-adic absolute value $|x|_p$ satisfies $|x|_p \leq p^n$ for some integer $n$.
%
%For $\xi\in \Q_p^{\times}$, we  have $|\xi|_p=p^m$  for some $m\in \Z$.
% By Lemma \ref{lem-fourierzero},    $\widehat{\nu} (\xi)=0$ if and only if  $\widehat{\nu} (-1/p^m)=0$,
% which is equivalent to
% \begin{align}\label{eq-Fourierzero}
%  \sum_{x\in E} \alpha_{x}\chi_{1/p^m}(x)=0.
%  \end{align}
%
%   Note that  the character   $\chi_{1/p^m}$ is  uniformly locally constant on the disks with radius $1/p^m$.
% Consider the map $\chi_{1/p^m}:E \to \{1, e^{2\pi i /p^{m+n}}, \cdots, e^{2\pi i (p^{m+n}-1)/p^{m+n}}\}$.
% Denote by $E_j=\{x\in E: \chi_{1/p^m}(x)=e^{2\pi i j/p^{m+n}} \}$.
% Equality \eqref{eq-Fourierzero} is equivalent  to
% \[
%\sum_{j=0} ^{p^{m+n}-1} a_j  e^{2\pi i j /p^{m+n}} =0,
%\]
% where $a_j=\sum_{x\in E_j} \alpha_x$.
%
%  For   $x\in E$,  assume that  $\chi_{1/p^m}(x)=e^{2\pi i j /p^{m+n}} $ for some  $0\leq j\leq p^{m+n}-1$.  By Lemma \ref{SchLemma},   $a_j=a_{j+p^{m+n-1} \mod p^{m+n}}$. If $a_j=0$, then there exist  $x^{\prime}\in E_j$ such that $x^{\prime}\neq x$.    Noting  that $\chi_{1/p^m}(x)= \chi_{1/p^m}(y)$ if and only if $|x-y|_p\leq 1/p^m$. Hence $|x-x^{\prime}|\leq 1/p^m$. On the other hand, $a_j\neq 0$, then there exist $x^{\prime}\in E$ such that
%  \[\chi_{1/p^m}(x^{\prime})=e^{2\pi i \frac{j+p^{m+n-1}}{p^{m+n}} },\]
%  which implies \[|x-x^{\prime}|_p=\frac{1}{p^{m-1}}.\]
  \end{proof}

\subsection{Bruhat-Schwartz distributions in $\Q_p$}\label{subsec2.4}
Here, we provide a concise overview of the theory based on the works of \cite{AKS10, Tai75, VVZ94}.
Let $\mathcal{E}$ represent the space of uniformly locally constancy functions. The space $\mathcal{D}$ of \textit{Bruhat-Schwartz test functions} is defined as functions that are uniformly locally constancy and have compact support. A test function $f\in \mathcal{D}$ is a finite linear combination of indicator functions of the form $1_{B(x,p^k)}(\cdot)$, where $k\in \mathbb{Z}$ and $x\in \Q_p$. The largest such number $k$ is denoted by $\ell:= \ell(f)$ and referred to as the \textit{parameter of constancy} of $f$. As $f\in \mathcal{D}$ has compact support, the minimal number $\ell':=\ell'(f)$ such that the support of $f$ is contained in $B(0, p^{\ell'})$ exists and is named the \textit{parameter of compactness} of $f$.

%Here we give a brief description of the theory following
%\cite{AKS10,Tai75,VVZ94}.
%  Let $\mathcal{E}$ denote the space of the uniformly locally constant functions.
%The space $\mathcal{D}$ of {\em Bruhat-Schwartz test functions} is, by definition, constituted of uniformly locally constant functions
% of compact support. Such a test function $f\in \mathcal{D}$ is a finite linear combination of indicator functions of the form $1_{B(x,p^k)}(\cdot)$, where $k\in \mathbb{Z}$ and $x\in \Q_p$. The largest of such numbers $k$
% is  denoted by $\ell:= \ell(f)$ and is  called the {\em parameter of constancy} of $f$. Since $f\in \mathcal{D}$ has compact support, the minimal number $\ell':=\ell'(f)$ such that the support of $f$ is contained in $B(0, p^{\ell'})$ exists and will be called the {\em parameter of compactness} of $f$.

Certainly, we have $\mathcal{D}\subset \mathcal{E}$. The space $\mathcal{D}$ is endowed with the topology of a topological vector space as follows: a sequence ${\phi_n }\subset \mathcal{D}$ is termed a \textit{null sequence} if there exist fixed integers $l$ and $l^{\prime}$ such that each $\phi_n$ is constancy on every ball of radius $p^l$, is supported by the ball $B(0,p^{l^{\prime}})$, and the sequence $\phi_n$ uniformly converges to zero.

A \textit{Bruhat-Schwartz distribution} $f$ on $\Q_p$ is defined as a continuous linear functional on $\mathcal{D}$. The value of $f$ at $\phi \in \mathcal{D}$ is denoted by $\langle f, \phi \rangle$. Notably, linear functionals on $\mathcal{D}$ are automatically continuous, facilitating the straightforward construction of distributions. Let $\mathcal{D}'$ denote the space of Bruhat-Schwartz distributions, equipped with the weak topology induced by $\mathcal{D}$.

A locally integrable function $f$ is treated as a distribution: for any $\phi \in \mathcal{D}$,
$$
\langle f,\phi\rangle=\int_{\Q_p} f\phi d\m.
$$

% Clearly, $\mathcal{D}\subset \mathcal{E}$. The space $\mathcal{D}$ is provided with a topology of topological vector space as follows: a sequence $\{\phi_n \}\subset \mathcal{D}$ is called a {\em null sequence} if there is a fixed pair of $l, l^{\prime}\in \mathbb{Z}$ such that each $\phi_n$ is constant on every ball of radius $p^l$ and is supported by the ball $B(0,p^{l^{\prime}})$ and the sequence $\phi_n$ tends uniformly to zero.
%
% A {\em Bruhat-Schwartz distribution} $f$ on $\Q_p$ is by definition a continuous linear functional on $\mathcal{D}$. The value of $f$ at $\phi \in \mathcal{D}$
% will  be denoted by $\langle f, \phi \rangle$.
% Note that linear functionals on $\mathcal{D}$ are automatically continuous. This property allows us to easily construct distributions. Denote by
% $\mathcal{D}'$ the space of  Bruhat-Schwartz  distributions. The space $\mathcal{D}'$ is equipped with the weak topology induced by  $\mathcal{D}$.
%
%  A locally integrable function $f$ is considered as a distribution:  for any $\phi \in \mathcal{D}$,
%$$
%\langle f,\phi\rangle=\int_{\Q_p} f\phi d\m.
%$$

%Let $E$ be a locally finite discrete subsets  $E$ in $\Q_p$ and   \[\nu=\sum_{x\in E} \alpha_{x}\delta_x,\quad  \alpha_{x}\in \Z\setminus\{0\},\]
%with  $\{\alpha_x, x\in E\}$ bounded be a discrete measure supported in $E$.

Let $E$ be a locally finite discrete subset in $\Q_p$, and let $\nu$ be a discrete signed measure supported in $E$ defined as
\begin{align}\label{eq-discretemeasure}
\nu = \sum_{x\in E} \alpha_{x}\delta_x, \quad \alpha_{x}\in \Z\setminus\{0\},
\end{align}
where $\{\alpha_x: x\in E\}$ is a bounded.
The discrete measure $\nu$ defined by  \eqref{eq-discretemeasure} is also a distribution:
for any $\phi \in \mathcal{D}$,
$$
\langle \nu,\phi\rangle= \sum_{x\in E} \alpha_x \phi(x).
$$
Here for each $\phi$, the sum is finite  because $E$ is locally finite and each ball contains at most a finite number of points in $E$.

Since the test functions in $\mathcal{D}$ are  uniformly locally constancy  and have compact support,
the following proposition is a direct consequence of Proposition \ref{Prop-compactConstant} or of the fact (see also \cite[Lemma 4]{Fan15}) that
\begin{equation}\label{FB}
\widehat{1_{B(c, p^k)}}(\xi) = \chi(-c \xi)\cdot p^{k} \cdot 1_{B(0, p^{-k})}(\xi).
\end{equation}

\begin{proposition}[\cite{Tai75}, Chapter II 3]
 The Fourier transformation  $f \mapsto \widehat{f}$ is a  homeomorphism  from $\mathcal{D}$ onto $\mathcal{D}$.
  \end{proposition}

Let $f \in \mathcal{D}'$ be a distribution in $\Q_p$.  A  point $x\in \Q_p$ is  called a {\em  zero} of  $f$ if there exists an integer $n_0$
  such that  $$ \langle f, 1_{B(y,p^{n})}\rangle=0, \quad \text {for all }  y\in B(x,p^{n_0})  \text{ and all integers  }  n\leq  n_0 \text.$$
Denote by  $\mathcal{Z}_f$ the set of all zeros of $f$.
Remark that
 $\mathcal{Z}_f$ is  the maximal open set $O$ on which $f$ vanishes, i.e.
 $\langle f, \phi\rangle=0$ for all $\phi \in \mathcal{D}$ such that the support of $\phi$ is contained in $O$.
 %$\operatorname{supp}(\phi)\subset O$.
 %Remark that  $x\in \mathcal{Z}_f $ if there exists an integer $n_0$
  %such that  $$ \langle f, 1_{B(x,p^{-n})}\rangle=0, \quad \text {for all integers  }  n\geq n_0.$$
 The {\em support}  of a distribution $f$ is defined as  the complementary set of  $\mathcal{Z}_f$ and is denoted by  $\operatorname{supp}(f)$.	

The Fourier transform of a distribution $f \in  \mathcal{D}'$ is a new distribution $\widehat{f}\in \mathcal{D}'$  defined by the duality
\[ \langle \widehat{f},\phi\rangle=\langle f,\widehat{\phi}\rangle, \quad \forall  \phi \in \mathcal{D}.\]
The Fourier transform $f\mapsto \widehat{f}$ is  a homeomorphism on  $\mathcal{D}'$ under the weak topology  \cite[Chapter II 3]{Tai75}.

%Remark that $n_0$  in the above proof depends only on the structure of $E$.  Define
%
%\begin{align}\label{def-n_E}
%n_E:=\max_{\substack {\lambda, \lambda^{\prime}\in E \\  \lambda\neq \lambda^{\prime}} } v_p(\lambda- \lambda^{\prime}).
%\end{align}
%According to the above proof of the second assertion of Proposition  \ref{zeroofE}, we immediately get
%\begin{align}\label{nE}
% \mathcal{Z}_{\widehat{\nu}}\subset   B(0,p^{n_E+1}).
%\end{align}

\subsection{Convolution and  multiplication of   distributions}
Denote $$
\Delta_k:=1_{B(0,p^k)}, \quad
\theta_k:=\widehat{\Delta}_k=p^k \cdot 1_{B(0,p^{-k})}.
$$
% The later is equal to $p^k1_{B(0,p^{-k})}$.
Let $f,g\in \mathcal{D}'$
 be two distributions. We define the {\em convolution} of $f$ and $g$ by
$$
\langle f*g,\phi \rangle =\lim\limits_{k\to \infty}	 \langle f(x), \langle g(\cdot),\Delta_k(x)\phi(x+\cdot) \rangle \rangle,
$$
if the limit exists for all $\phi \in \mathcal{D}$.

{%\color{red}
\begin{proposition} [\cite{AKS10}, Proposition  4.7.3  ] If $f\in \mathcal{D}^{\prime}$, then  $f*\theta_k\in \mathcal{E}$
with  the
parameter of constancy at least  $-k$.
\end{proposition}
}

%Remark  that $f*\theta_k\in \mathcal{E}$  for any distribution  $f\in \mathcal{D}' $ ().
We define the
{\em  multiplication} of $f$ and $g$ by
$$
\langle f\cdot g,\phi \rangle =\lim\limits_{k\to \infty}	 \langle g, ( f*\theta_k)\phi  \rangle,
$$
if the limit exists for all $\phi \in \mathcal{D}$.        The above definition of convolution is
compatible with the usual convolution of two integrable functions and the definition of multiplication is compatible with  the usual  multiplication of two locally integrable functions.

The following proposition shows that both the convolution and the multiplication are commutative when they are well defined and the convolution of two distributions is well defined if and only if the multiplication % \marginpar{what is "order"?}
of their Fourier transforms is well defined.

\begin{proposition} [\cite{VVZ94}, Sections 7.1 and 7.5] \label{Conv-Mul} Let $f, g\in \mathcal{D}'$ be two distributions. Then\\
\indent {\rm (1)} \  If $f*g$ is well defined, so is $g*f$ and  $f*g=g*f$.\\
\indent {\rm (2)} \ If $ f\cdot g$ is well defined, so is  $g\cdot f$
 and  $ f\cdot g= g\cdot f$.\\
 \indent {\rm (3)} \ $f*g$ is well defined  if and only   $\widehat{f}\cdot \widehat{g}$ is well defined. In this case, we have
$
\widehat{f*g}=\widehat{f}\cdot\widehat{g}$ and
$ \widehat{f\cdot g}=\widehat{f}*\widehat{g}.
$
\end{proposition}

\section{Proof of Theorem  \ref{thm-tilingperiodic}}
%\begin{Definition}
	%A Borel set $\Omega \subset {\Q}_{p} $ is said to be \textbf{\textit{almost compact open}}  if there exists a compact open set $\Omega '$ such that $\mathfrak{m}\left ( \Omega \setminus \Omega ' \right ) =\mathfrak{m}\left ( \Omega' \setminus \Omega  \right )=0$.

In this section,
we will   study the following equation
\begin{equation}\label{EQ}
 f*\nu = w
\end{equation}
where  $ f\in L^1(\Q_p)$ with  $\int_{\Q_p}fd\m>0$, and
$\nu$ is  a locally finite  signed measure defined  by \eqref{eq-discretemeasure}  in $\Q_p$. %and $w>0$ is a positive real number.
Our discussion is based on the functional equation \[\widehat{f}\cdot\widehat{\nu}=w\delta_0,\] which is implied by
 $f*\nu=w$ (see Proposition \ref{Conv-Mul}).

 \subsection{Zeros of the Fourier transform  of  the  measure $\nu$}
Let  $\nu$  be a discrete signed  measure defined by  \eqref{eq-discretemeasure}. The support  of $\nu$ is the discrete set $E$. Denote by
$\nu_n$ the restriction of $\nu$ on the ball $B(0,p^n)$,i.e.
\[ \nu_n=\sum_{x\in E\cap B(0,p^n) } \alpha_{x}\delta_x. \]
Consider $\nu$  as a  distribution on  $\Q_p$.
For simplicity, denote by $E_n=E\cap B(0,p^n) $ the restriction of $E$ on the ball $B(0, p^n)$.
The following proposition characterizes the structure of  $ \mathcal{Z}_{\widehat{\nu}}$, the set of zeros of the
Fourier transform of $\nu$.  It is bounded and is a union of spheres centered at $0$.
Recall that  the support  of $\nu$ is a locally finite subset. Define
\begin{align}\label{def-nv}
n_{\nu}:=\inf_{\lambda \in E  } \max_{\substack {\lambda^{\prime}\in E\\  \lambda\neq \lambda^{\prime}} } v_p(\lambda- \lambda^{\prime}).
\end{align}
Remark that $n_{\nu}$ can be $-\infty$.

\begin{proposition}\label{zeroofE}
Let $\nu$ be a  discrete measure   defined by  \eqref{eq-discretemeasure} in $\Q_p$.\\
\indent {\rm (1)}
If $\xi\in  \mathcal{Z}_{\widehat{\nu}}$,  then $S(0,|\xi|_p)\subset   \mathcal{Z}_{\widehat{\nu}}$.\\
\indent {\rm (2)} The set $ \mathcal{Z}_{\widehat{\nu}}$  is contained in   $B(0,p^{n_\nu+1})$.
%Let $\nu$ be a discrete measure defined by \eqref{eq-discretemeasure} in $\Q_p$. \begin{enumerate}
%\item If $\xi \in \mathcal{Z}{\widehat{\nu}}$, then $S(0,|\xi|p) \subset \mathcal{Z}{\widehat{\nu}}$.
%\item The set $\mathcal{Z}{\widehat{\nu}}$ is bounded, and it satisfies
%\begin{align}\label{nE}
%\mathcal{Z}_{\widehat{\nu}} \subset B(0, p^{n_E+1}).
%\end{align}
%\end{enumerate}
\end{proposition}
\begin{proof}
	%Since $\ZE$ and $S(0,|\xi|)$ are open,  $\xi\in \ZE$ implies that  $$ B(\xi, p^{-n'})\subset \ZE \cap S(0,|\xi|) \quad \text {for some $n'>-\log_p |\xi|$ }.$$
 First remark that by using (\ref{FB}) we get  %\marginpar{put the reason before}
	\begin{equation}\label{FofMu}
		\langle \widehat{\nu}, 1_{B(\xi,p^{-n})} \rangle
		=\langle \nu, \widehat{1_{B(\xi,p^{-n})}} \rangle
		%&=\langle \mu_\Lambda, p^{n'}1_{B(0,p^{-n'})}\overline{\chi_{\xi}} %\rangle
		=p^{-n}\sum_{x\in E_n} \alpha_x \cdot \overline{\chi(\xi x)}=p^{-n} \widehat{\nu_n}(\xi).
	\end{equation}
	This expression will be used several times.
	
	\indent {\rm (1)}
By definition,  $\xi\in  \mathcal{Z}_{\widehat{\nu}}$
implies that
there exists an integer $n_0$ such that
 $$\langle \widehat{\nu}, 1_{B(\xi,p^{-n})} \rangle=0 , \quad \forall~ n\geq n_0. $$
By (\ref{FofMu}), this is equivalent to
 \begin{equation}\label{zeross}
 \sum_{x\in E_n} \alpha_x \cdot \chi(\xi x)=0 ,\quad \forall~ n\geq n_0.
 \end{equation}	

For any  $\xi^{\prime}\in S(0,|\xi|_p)$,  we have $\xi' = u \xi$ for some $u\in \mathbb{Z}_p^\times$.  By Corollary \ref{corollary2} and  the equality (\ref{zeross}), we obtain
	$$
	\sum_{x\in E_n}\alpha_x \cdot \chi(\xi^{\prime}x )= \sum_{x\in E_n}\alpha_x \cdot  \chi(\xi x)=0,\quad \forall~ n\geq n_0.
	$$
	 Thus, again by (\ref{FofMu}), $\langle \widehat{\nu}, 1_{B(\xi',p^{-n})} \rangle=0 $ for $ n\geq n_0$. % , i.e. $\xi' \in \ZE$.
	 We have thus proved $S(0,|\xi|_p)\subset \mathcal{Z}_{\widehat{\nu}}$.

 \indent {\rm (2)} We distinguish two cases based on the value of $n_{\nu}$ : $n_\nu=-\infty$  or $n_\nu>-\infty$. 
 
{ \bf Case 1: $n_\nu=-\infty$.} Fix $\xi \in  \Q_p\setminus\{0\}$, since $n_\nu=-\infty$,  there exists  $x_0\in E$ such  that 
 $$\forall x \in   E \setminus \{x_0\}, \quad |x-x_0|_p\geq 1/|\xi|_p.$$
 We are going to show that  $ \mathcal{Z}_{\widehat{\nu}}\subset \{0\}$, which is equivalent to that  for each  $\xi \in \Q_p\setminus\{0\},$ 
 \[ \langle \widehat{\nu}, 1_{B(\xi,p^{-n})} \rangle \neq 0\]
 for sufficient large $n$ such that $x_0\in B(0,p^n)$.   In fact, if this is not the case, then by (\ref{FofMu}),  $\widehat{\nu_n}(\xi)=0.$
 Thus by Lemma \ref{lem-necessary}, we can find $x\in E_n$ 
 such that \[ |x-x_0|=p/|\xi|_p<1/|\xi|_p,\]
 a contradiction.
 
{\bf  Case 2: $n_\nu>-\infty$.}
 Hence, there exists  a  point    $x_0 \in E$ such that  
 \[n_\nu= \max_{\substack {x\in E\\  x\neq x_0} } v_p(x- x_0).\]
 It follows that 
$$\forall x \in   E \setminus \{x_0\}, \quad |x-x_0|_p\geq p^{-n_\nu}.$$
We are going to show that  $ \mathcal{Z}_{\widehat{\nu}}\subset B(0,p^{n_\nu+1})$, which is equivalent to that $\xi \not \in  \mathcal{Z}_{\widehat{\nu}}$ when $|\xi|_p\geq p^{n_\nu+2}$. To this end, we will prove that for all integer $n$ large enough such that $x_0 \in B(0, p^n)$, we have
$$\langle \widehat{\nu}, 1_{B(\xi,p^{-n})} \rangle \neq 0.$$
In fact, if this is not the case, then by (\ref{FofMu}),  $\widehat{\nu_n}(\xi)=0.$
Thus, by Lemma \ref{lem-necessary}, for the given $x_0$, we can find $x \in E_n$, such that \[|x-x_0|_p\leq p/|\xi|_p <p^{-n_{\nu}},\]  a contradiction.

\end{proof}
\subsection{Proof of Theorem \ref{thm-tilingperiodic}}

For a function $g: \Q_p \rightarrow \mathbb{R}$, denote
$$
   \mathcal{N}_g :=\{x \in \Q_p: g(x)=0\}.
$$
If $g\in C(\Q_p)$ is a continuous function, then $\mathcal{N}_g $ is a closed set and $ \mathcal{Z}_g$ is the set of interior points of $\mathcal{N}_g $.
 %  when $g$ is considered  as a distribution in $\mathcal{D}'$.
  However, the support of $g$ as a continuous function is equal to the support of $g$ as  a distribution.

\begin{proposition}[{\cite[Proposition 3.2]{FFLS}}]\label{mainlem}
	Let $g\in C(\Q_p)$ be a continuous function and let $G\in \mathcal{D}'$ be a distribution. Suppose that the product $H=g\cdot G$ is well defined. 	Then
		$$\mathcal{Z}_{H} \subset \mathcal {N}_{g} \cup \mathcal{Z}_{G}.$$
		\end{proposition}
For $f\in L^1(\Q_p)$, it is evident that $\widehat{f}$ is a continuous function. Consequently, we obtain the following immediate consequence.
\begin{corollary}\label{cor:3.3}
If $f*\nu=w$ with $f\in L^1(\mathbb{Q}_p)$, then
		$\Q_p \setminus \{0\} \subset \mathcal {N}_{\widehat{f}} \cup \mathcal{Z}_{\widehat{\nu}},$
		which is equivalent to
	\begin{align}\label{mainprop}
	\{\xi\in\Q_p: \widehat{f}(\xi) \neq 0  \} \setminus \{0\} \subset \mathcal{Z}_{\widehat{\nu}}.
	\end{align}
\end{corollary}

\begin{proof}[Proof of Theorem  \ref{thm-tilingperiodic}]
By Proposition \ref{zeroofE}, we have $ \mathcal{Z}_{\widehat{\nu}} \subset B(0,p^{n_v+1})$. According to Corollary \ref{cor:3.3}, ${ \rm supp }(\widehat{f}) \subset B(0,p^{n_v+1})$, which is bounded and, therefore, compact. Following Proposition \ref{Prop-compactConstant}, we conclude that $f$ is uniformly locally constancy.
\end{proof}

\section{Density of  $\nu$ and the proof of Theorem \ref{thm:card}}
%\begin{theorem}Let $f\in L^1(\Q_p)$ be a
%\end{theorem}

%Suppose that $(f, \nu)$ is a solution of (\ref{EQ}). We will investigate
%the density of $E$ and even   the distribution of $E$,  and the supports of the Fourier transforms
%$\widehat{\mu_E}$ and $\widehat{f}$.

We say that the discrete measure $\nu$ has a {\em bounded density} if the following limit exists for some $x_0\in \Q_p$,	
$$
D(\nu):=\lim\limits_{k\to \infty}\frac{ \nu(B(x_0,p^k))}{\m (B(x_0,p^k))},
$$
%where $|A|$ denotes the Haar measure of a Borel set $A$.
which is called {\em the density} of the discrete measure  $\nu$. Actually, if the limit exists for some $x_0 \in \Q_p$, then
it exists for all $x\in \Q_p$ and the limit is independent of $x$. In fact, for any $x_0, x_1\in \Q_p$, when $k$ is  large enough such that  $|x_0-x_1|_p<p^k$, we have $B(x_0, p^k)=B(x_1, p^k)$.

  The following theorem gets together some properties of the solution $(f, \nu)$ of the equation \eqref{EQ}, which will be proved in this section.

\begin{theorem} \label{M1} Let $f\in L^1(\mathbb{Q}_p)$ with  $\int_{\mathbb{Q}_p}fd\m>0$,
	and $\nu$ be a discrete measure defined by \eqref{eq-discretemeasure}  with locally finite support  in $\Q_p$. %and $w>0$ be a positive real number.
	Suppose that the equation {\rm (\ref{EQ})} is satisfied by $f$ and $\nu$ for some $w\geq 0$. Then the following statements hold.\\
	%\indent {\rm (1)}\
	   %    $\mbox{\rm supp} \widehat{\mu_E} \subset \{0\}\cup \mathcal{N}_{\widehat{f}}$, or equivalently $\widehat{\mu_E}$ vanishes in $\mathcal{N}_{\widehat{f}}^{c}\setminus \{0\}$. \\
	       		%\indent {\rm (1)}\
	       		%The support of $\widehat{f}$ is compact.\\
	      	%$\mbox{\rm supp} \widehat{f}$ is compact.\\
	%\indent {\rm (2)}\
	   %   $E$ has a bounded density $D(E)$ and  $D(E)= 1/\int f(x) dx$.\\ %$D(E)\cdot \int f(x) dx = 1$.\\
	   \indent {\rm (1)} The set  $\mathcal{Z}_{\widehat{\nu}}$  is bounded  and it is the union of  the  punctured ball and some spheres with the same center $0$. \\
	\indent {\rm (2)}\ The density $D(\nu)$ exists and equals to
	$1/\int_{\Q_p} f d\m$. Furthermore,
	there exists an integer $n_f\in \mathbb{Z}$
	%, which depends  on $f$,
	such that  for all integers $n \geq n_f$ we have
	       $$
	       \forall \xi \in \Q_p, \quad \nu \big(B(\xi, p^{n})\big) = p^{n} D(\nu).
	        $$

\end{theorem}

Theorem \ref{M1} (1) and  will be proved in $\S 4.1$, the distribution of $\nu$ will be discussed in $\S 4.2$ and the equality
$D(\nu)=w/ \int_{\Q_p} f d \m$ will be proved in $\S 4.3$.

\subsection{ Structure of  $\mathcal{Z}_{\widehat{\nu}}$}

Our discussion is based on the functional equation $\widehat{f}\cdot\widehat{\nu}= w \cdot \delta_0$, which is implied by
 $f*\mu_\nu=w$ (see Proposition \ref{Conv-Mul}).

Notice that $\widehat{f}(0)=\int_{\Q_p} f d \m>0$ and  that $\widehat{f}$ is a continuous function.
It follows that  there exists a small ball  where $\widehat{f}$ is nonvanishing.
Let
\begin{align} \label{nf}
	n_{f}:=\min \{n\in\mathbb{Z}: \widehat{f}(x)\neq 0,\text{ if }  x\in B(0,p^{-n}) \}.
\end{align}

%\subsection{Compactness of $\mbox{\rm supp} \widehat{f}$}
\begin{proposition}\label{prop:4.2} Let $f\in L^1(\mathbb{Q}_p)$ with  $\int_{\mathbb{Q}_p}fd\m>0$,
	and $\nu$ be a discrete measure defined by \eqref{eq-discretemeasure}  with locally  a finite support  in $\Q_p$.  Then,
\begin{equation} \label{eq:4.2}
 B(0,p^{-n_f})\setminus \{0\}\subset \mathcal{Z}_{\widehat{\nu}}.
\end{equation}
\end{proposition}

\begin{proof}   Actually,  \eqref{eq:4.2} is an
	immediately consequence of  \eqref{mainprop}.
\end{proof}

By Proposition \ref{zeroofE} and \ref{prop:4.2},  the set  $\mathcal{Z}_{\widehat{\nu}}$  is bounded  and it is the union of  the  punctured ball $B(0,p^{-n_f})$ and some spheres with the same center $0$.

\subsection{Distribution of $\nu$}

 The discrete measure $\nu$ involved in the equation (\ref{EQ}) shares the following uniform distribution property.

\begin{proposition}\label{Structure}
The measure $\nu$ of the ball $B(\xi, p^{n_f})$ is independent of $\xi \in \mathbb{Q}_p$. Consequently, the measure  $\nu$ admits a bounded density $D(\nu)$. Moreover, for all integers $n \geq n_f$, we have
\begin{equation}
    \label{numberE}
    \forall \xi \in \mathbb{Q}_p, \quad \nu\big(B(\xi, p^{n})\big) = p^n D(\nu).
\end{equation}

%	The measure  $\nu$ of the ball  $B(\xi,p^{n_f})$  is independent  of $\xi\in \mathbb{Q}_p$.
%	Consequently,  the   set  $\nu$ admits  a bounded density $D(E)$.  Moreover,  for all integers $n\geq n_f$, we have
%	\begin{equation}  \label{numberE}
% \forall \xi \in \Q_p, \quad	\nu(B(\xi,p^{n_f}))=  p^{n} D(E).
%	\end{equation}
\end{proposition}
\begin{proof}

%Denote by
%$\nu^{\xi}_n$ the restriction of $\nu$ on the ball $B(\xi,p^n)$,i.e.
%\[ \nu^{\xi}_n=\sum_{x\in E\cap B(\xi,p^n) } \alpha_{x}\delta_x. \]
%And write $\nu_{n}: =\nu^{\xi}_{n}$ when $\xi=0$.
%
%	For simplicity, we denote
%	 $$E_n^\xi:=E\cap B(\xi,p^n),$$
%%	where  $n\in \mathbb{Z}$ and $\xi \in \Qp$.
%	and write $E_n:=E\cap B(0,p^n)$  when $\xi=0$. It suffices to prove
%	 $$\forall \xi \in \Q_p,  \quad\Card(E_{n_f})= \Card(E_{n_f}^{\xi}).$$
	
For any given $\xi \in \Q_p$, let $k=-v_p(\xi)$. If $k\leq n_f$, then $B(0,p^{n_f})=B(\xi,p^{n_f})$. So obviously $\nu\big(B(0,p^{n_f})\big)= \nu\big(B(\xi,p^{n_f})\big)$.

Now we suppose $k> n_f$. Then, consider any  $\eta$ satisfying $$B(\eta,p^{-k})\subset B(0,p^{-n_f})\setminus \{0\}.$$
By \eqref{eq:4.2} in Proposition \ref{prop:4.2}, we have   $\langle \widehat{\nu}, 1_{B(\eta,p^{-k})} \rangle =0$, which
by (\ref{FofMu}), is equivalent to
\begin{equation}\label{eq:4.4}
\sum_{x\in E_k} \alpha_x \cdot \chi(\eta x)=0,
%\sum_{\lambda\in E_k}{\chi(\eta \lambda)}=0.
\end{equation}
where $E_k= {\rm supp} \nu  \cap B(0,p^k).$
Taking  $\eta= p^{k-1}$ in \eqref{eq:4.4}, we have
$
\sum_{x\in E_k} \alpha_x \cdot{\chi(p^{k-1}x)}=0.
$
Observe that  $$B(0,p^k)= \bigsqcup_{i=0}^{p-1} B(ip^{-k},p^{k-1})$$ and that the function  $\chi(p^{k-1}\cdot)$ is constant  on each ball of radius $p^{k-1}$.
So we have
$$
0=\sum_{x\in E_k}\alpha_x  \cdot {\chi(p^{k-1}x)}=\sum_{i=0}^{p-1} {\chi\Big(\frac{i}{p}\Big)}\nu\big(B(ip^{-k},p^{k-1})\big).
$$
Applying
Lemma \ref{SchLemma}, we obtain
\begin{equation}
\label{ud1}
\nu \big(B(ip^{-k},p^{k-1})\big) =\nu\big(B(jp^{-k},p^{k-1})\big),
\quad \forall \ 0\le i,  j \le p-1.
\end{equation}

 Similarly,   taking $\eta=p^{k-2}$ in  (\ref{eq:4.4}), we have

$$
0=\sum_{0\le i, j \le p-1}{\chi\Big(\frac{i}{p^2}+\frac{j}{p}\Big)}  \nu \big(B(\frac{i}{p^{k}}+\frac{j}{p^{k-1}}, p^{k-2})\big).
$$
Again, Lemma \ref{SchLemma} implies
\[ \nu\big(B(\frac{i}{p^{k}}+\frac{j}{p^{k-1}}, p^{k-2}) \big)=\nu\big(B(\frac{i}{p^{k}}+\frac{m}{p^{k-1}}, p^{k-2})\big),  \quad \forall \ 0\le i,  j,m \le p-1. \]
%{\rm Card}\Big(E_{k-2}^{\frac{i}{p^{k}}+\frac{j}{p^{k-1}}}\Big)={\rm Card}\Big(E_{k-2}^{\frac{i}{p^{k}}+\frac{m}{p^{k-1}}}\Big)\quad \forall \ 0\le i,  j,m \le p-1.$$

Since $$\sum_{j=0}^{p-1} \nu \big(B(\frac{i}{p^{k}}+\frac{j}{p^{k-1}}, p^{k-2})\big)= \nu\big(B(\frac{i}{p^k}, p^{k-1})\big),$$ by (\ref{ud1}), we get
$$
\nu\big(B(\frac{i}{p^{k}}+\frac{j}{p^{k-1}}, p^{k-2}) \big)=\nu\big(B(\frac{l}{p^{k}}+\frac{m}{p^{k-1}}, p^{k-2})\big) \quad   \forall\  0\le i, j, l,m \le p-1.
$$
%We continue this way  for all $\xi= p^k, p^{k-1}, \dotsc, p^{n_f}$.
%In fact, we argue by induction on $k$ in order to conclude.,

We continue these arguments    for all $\eta= p^{k-1}, \dotsc, p^{n_f}$.
By induction, we have
\begin{align}\label{equalnumber}
\nu\big(  B(\xi_1, p^{n_f})\big)=\nu\big(  B(\xi_2, p^{n_f})\big), \quad  \forall \xi_1,\xi_2 \in D(0,p^k).
%{\rm Card}(E_{n_f}^{\xi_1})
%={\rm Card}(E_{n_f}^{\xi_2}), \quad   \forall \  \xi_1,  \xi_2 \in \mathbb{L}_{n_f} \cap B(0,p^k).
\end{align}
%Since $|\xi|_p=p^{k}$, there exists $\xi^{\prime}\in \mathbb{L}_{n_f} \cap B(0,p^k)$, such that
%$B(\xi_1, p^{n_f})=B(\xi_2, p^{n_f}).$
%$E_{n_f}^{\xi}=E_{n_f}^{\xi^{\prime}}$.
Thus  by (\ref{equalnumber}),
$$ \nu \big(B(\xi, p^{n_f})\big)=\nu \big( B(0, p^{n_f} \big).$$

The formula (\ref{numberE}) follows immediately because
each ball of radius  $p^n$ with $ n\ge n_f$ is a disjoint union of  $p^{n-n_f}$ balls of radius $p^{n_f}$ so that
$$
   \nu \big( B(0,p^n)\big) = p^{n-n_f}  \cdot \nu \big( B(0,p^{n_f})\big) .
$$
%Divide by $p^N$ and let $N$ tend to infinity, we get
%that the limit defining $D(E)$ exists and $D(E)= p^{-n}{\rm Card} E_n$.
\end{proof}

\subsection {Equality $D(\nu)=w/ \int_{\Q_p} f d\m$}
%Now we are going to calculate  density of $E$.
\begin{proposition}\label{fisrtprop}
%The set $E$ has bounded density and its
The  density $D(\nu)$ of $\nu$ satisfies
 $$D(\nu)=\frac{w}{\int_{\Q_p}f(x)dx}.
 $$
\end{proposition}

\begin{proof}
%We first give a remark that for any integer $N$ and any point $e\in B(0, p^N)$ we have
%\begin{equation}\label{remark1}
%  \int_{B(0,p^N)} f(x-e) dx = \int_{B(0,p^N)} f(x) dx
%  \end{equation}
%  \begin{equation}\label{remark2}
%  \int_{B(0,p^N)^c} f(x-e) dx = \int_{B(0,p^N)^c} f(x) dx.
%\end{equation}
%The first equality is just because $B(0,p^N) = B(e,p^N)$ and the second one results from the first and %from the
%the invariance by translation of the Haar measure.
By the integrability of $f$, the quantity
$$
\epsilon_n := \int_{\Q_p\setminus B(0, p^n)} f(x) dx
$$
tends to zero  as  the integer $n\to \infty$.
%we shall make use of the same notation as in the  proof of Lemma \ref{bounded}.
%	Let $N \ge 1$.
	Integrating the equality (\ref{EQ}) over the ball $B(0, p^n)$, we have
\begin{eqnarray}\label{ee}
  w \cdot  \m (B(0,p^n))
		=\sum_{\lambda \in E} \int_{B(0,p^n)}\alpha_{\lambda}\cdot  f(x-\lambda)dx.
\end{eqnarray}
Now we split the sum in (\ref{ee}) into two parts, according to $\lambda \in E\cap B(0, p^n)$ or $\lambda \in E\setminus B(0, p^n)$.  Denote $I:=\int_{\Q_p} f d\m$.
For $\lambda \in B(0, p^n)$,  we have
\begin{eqnarray*}
\int_{B(0,p^n)} f(x-\lambda)dx =	\int_{B(0,p^n)} f(x)dx
	= I - \epsilon_n.
\end{eqnarray*}
%where $I$ is the integral of $f$ over $\Q_p$.
It follows that
\begin{equation}\label{ee1}
\sum_{\lambda \in E\cap B(0, p^n)}	\int_{B(0,p^n)} \alpha_{\lambda} \cdot f(x-\lambda)dx
=   \nu\big(  B(0, p^n)\big)\cdot (I-\epsilon_n).
	%\Big(\int_{\Q_p} f(x)dx - \epsilon_N\Big).
\end{equation}

Notice that
\begin{equation}\label{ee2}
    \int_{B(0,p^n)} f(x-\lambda)dx
    = \int_{B(-\lambda,p^n)} f(x)dx.
\end{equation}
For $\lambda \in E\setminus B(0, p^n)$, the ball  $B(-\lambda,p^n)$ is contained in $\Q_p\setminus B(0, p^n)$.
We partition the support  $\Q_p\setminus B(0, p^n)$  into $B_j$'s such that each $B_j$ is  a ball
of radius $p^n$. Thus the integrals in (\ref{ee2}) for the $\lambda$'s in the same $B_j$ are equal.
 Let $\lambda_j$ be a representative of $P_j$. Then we have
\begin{eqnarray}
\sum_{\lambda \in E\setminus B(0,p^n)} \int_{B(0,p^n)}  \alpha_{\lambda} \cdot f(x-\lambda)dx
 =  \sum_j \nu(B_j) \cdot \int_{B(-\lambda_j,p^n)} f(x)dx. \nonumber
 \end{eqnarray}
 However, by (\ref{numberE}) in Proposition \ref{Structure},  $\nu (B_j) = D(\nu) \cdot \m\big(B(0, p^n)\big)$  if $n \geq n_f$. Thus, for each  integer $n\geq n_f$,
 \begin{align}
 \sum_{\lambda\in E\setminus B(0,p^n)} \int_{B(0,p^n)}\alpha_{\lambda} \cdot  f(x-\lambda)dx
       & =  D(\nu) \cdot \m\big(B(0, p^n)\big) \sum_j \int_{B_j} f(x)dx  \nonumber \\
       & =  D(\nu)\cdot  \m\big(B(0, p^n)\big) \int_{\Q_p \setminus B(0, p^n)} f(x) dx  \nonumber  \\
       &=   D(\nu)\cdot  \m\big(B(0, p^n)\big)\cdot \epsilon_n. \label{ee3}
\end{align}

Thus from (\ref{ee}), (\ref{ee1}) and (\ref{ee3}),  we finally get
	$$
	\left| w\cdot \m\big(B(0, p^n)\big) - \nu\big( B(0, p^n)\big)\cdot I \right|
      \le 2 D(\nu)  \cdot \m(B(0, p^n))\cdot \epsilon_n.
	$$
We conclude by dividing $\m\big(B(0, p^n)\big)$ and then letting $n\to \infty$.
\end{proof}

 \subsection{Proof of Theorem \ref{thm:card} }
 \begin{proof}[Proof of Theorem \ref{thm:card}]
 Consider the convolution equation $f*\nu=1$, where $f\in L^{1}(\Q_p)$ is a non-negative,  $\nu=\sum_{t\in T}\delta_t$ and $T$ is a discrete subset in $\Q_p$.
 By Proposition \ref{Structure},    for  integers $n \geq n_f$, we have
 \[ \#(B(x,p^n)\cap T) =\#(B(y,p^n)\cap T), \quad  \forall x, y \in \mathbb{Q}_p.\]

 On the other hand,  if   \[ \#(B(x,p^n)\cap T) = \#(B(y,p^n)\cap T), \quad  \forall x, y \in \mathbb{Q}_p.\]  for some $n\geq 0$. Take \[f=\frac{1_{B(0,p^n)}}{\#(B(x,p^n)\cap T)}. \]
 Then $f*\nu=1$ with $\nu=\sum_{t\in T}\delta_t$.
%\begin{equation}
%    \label{numberE}
%    \forall \xi \in \mathbb{Q}_p, \quad \nu\big(B(\xi, p^{n_f})\big) = p^n D(\nu).
%\end{equation}
 \end{proof}

\section{The structure of tiles on $\mathbbm {\Z}/ p^{n}q\mathbbm {\Z}$ and $\Z/p^n\Z\times\Z/p\Z$}
In this section, we primarily utilize the results established in \cite{FFS} and \cite{FKL} to characterize the structure of tiles on
$\mathbbm {\Z}/ p^{n}q\mathbbm {\Z}$. The structural properties of tiles in $\mathbb{Z}/p^n\mathbb{Z}$ are characterized by $p$-homogeneity, as demonstrated in \cite{FFS}.

\subsection{ $p$-homogeneity  of tiles in $\Z/p^n\Z$}

Let $nn$ be a positive integer. To any finite sequence $t_0 t_1 \cdots t_{n-1} \in \{0, 1, \dots, p-1\}^{n}$, we associate the integer
\[
c = c(t_0 t_1 \cdots t_{n-1}) = \sum_{i=0}^{n-1} t_i p^i \in \{0, 1, \dots, p^{n} - 1\}.
\]
This establishes a bijection between ${\Z}/p^{n}{\Z}$ and $\{0, 1, \dots, p-1\}^{n}$, which we consider as a finite tree, denoted by ${\mathcal T}^{(n)}$ (see Figure~\ref{fig:1}).

The set of vertices of ${\mathcal T}^{(n)}$ is the disjoint union of the sets ${\Z}/p^\gamma{\Z}$, for $0 \le \gamma \le n$. Each vertex, except the root, is identified with a sequence $t_0 t_1 \cdots t_{\gamma-1}$ in ${\Z}/p^\gamma{\Z}$, where $0 \le \gamma \le n$ and $t_i \in \{0, 1, \dots, p-1\}$. The set of edges consists of pairs $(x, y) \in {\Z}/p^\gamma{\Z} \times {\Z}/p^{\gamma+1}{\Z}$ such that $x \equiv y \pmod{p^n}$, where $0 \le \gamma \le n-1$.
Each point $c$ of ${\Z}/p^n{\Z}$ is identified with the boundary point $\sum_{i=0}^{n-1} t_i p^i \in \{0, 1, \dots, p^{n} - 1\}$ of the tree.

Each subset $C \subset {\Z}/p^n{\Z}$ determines a subtree of ${\mathcal T}^{(n)}$, denoted by ${\mathcal T}_C$, which consists of the paths from the root to the boundary points in $C$. 

For each $0 \le \gamma \le n$, we denote by
\[
C_{\bmod p^\gamma} := \{x \in \{0, 1, \dots, p^\gamma-1\} : \exists y \in C \text{ such that } x \equiv y \pmod{p^\gamma}\}
\]
the subset of $C$ modulo $p^\gamma$.

The set of vertices of ${\mathcal T}_C$ is the disjoint union of the sets $C_{\bmod p^\gamma}$, for $0 \le \gamma \le n$. The set of edges consists of pairs $(x, y) \in C_{\bmod p^\gamma} \times C_{\bmod p^{\gamma+1}}$ such that $x \equiv y \pmod{p^\gamma}$, where $0 \le \gamma \le  n-1$.

For vertices $u \in C_{\bmod p^{\gamma+1}}$ and $s \in C_{\bmod p^\gamma}$, we call $s$ the  \emph{parent} of $u$ or $u$ the \emph{descendant} of $s$ if there exists an edge between $s$ and $u$.

Now, we proceed to construct a class of subtrees of ${\mathcal T}^{(n)}$. Let $I$ be a subset of $\{0, 1, \dots, n-1\}$, and let $J$ be its complement. Thus, $I$ and $J$ form a partition of $\{0, 1, \dots, n-1\}$, and either set may be empty.

We say a subtree ${\mathcal T}_C$ of ${\mathcal T}^{(n)}$ is of ${\mathcal T}_{I}$-form if its vertices satisfy the following conditions:

\begin{enumerate}
    \item If $i \in I$ and  $t_0 t_1\dots t_{i-1} $ is given, then $t_i$ can take any value in $\{0, 1, \dots, p-1\}$. In other words, every vertex in $C_{\bmod p^{i}}$ has $p$ descendants.
    \item If $i \in J$ and $t_0 t_1 \dots t_{i-1}$ is given, we fix a value in $\{0, 1, \dots, p-1\}$ that $t_i$ must take. That is, $t_i$ takes only one value from $\{0, 1, \dots, p-1\}$, which depends on $t_0 t_1 \dots t_{i-1}$. In other words, every vertex in $C_{\bmod p^{i}}$ has one descendant.
\end{enumerate}
  \begin{figure}
  	\centering
  	\includegraphics[width=0.7\linewidth]{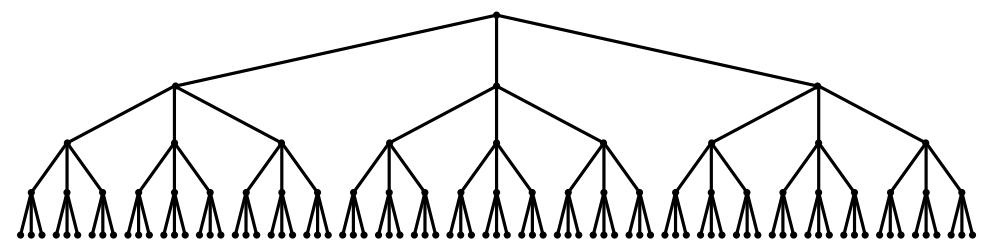}
  	\caption{The set ${\Z}/3^4{\Z}=\{0,1,2,\cdots,80\}$ is considered as a tree ${\mathcal T}^{\left(4\right)} $. }
  	\label{fig:1}
  \end{figure}
  \begin{figure}
  	\centering
  	\includegraphics[width=0.7\linewidth]{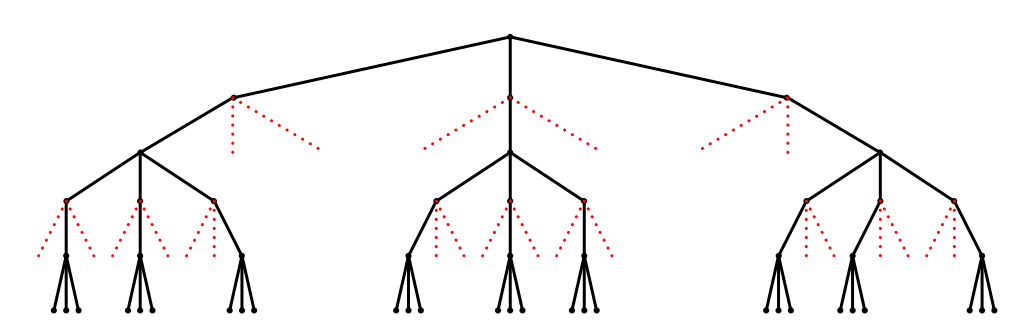}
  	\caption{For $p=3$, a ${{\mathcal T}_{I,J }}$-form tree with $n=5$, $I=\{0,2,4\} ,J=\{1,3\}$.}
  	\label{fig:2}
  \end{figure}

Note that such a subtree depends not only on $I$ and $J$ but also on the specific values assigned to $t_i$ for $i \in J$. A ${\mathcal T}_{I}$-form tree is called a finite \emph{$p$-homogeneous tree.} An example of a ${\mathcal T}_{I}$-form tree is shown in Figure~\ref{fig:2}.

 A set $C \subset {\Z}/p^n{\Z}$ is said to be  \emph{$p$-homogeneous} subset of ${\Z}/p^n{\Z}$ with  the \emph{branched level set} $I$ if the corresponding tree ${\mathcal T}_C$ is $p$-homogeneous of form  ${\mathcal T}_{I}$.

If we consider $C \subset \{0, 1, \dots, p^n - 1\}$ as a subset of ${\Z}_p$, then the tree ${\mathcal T}_C$ can be identified with the finite tree determined by the compact open set
\[
\Omega = \bigsqcup_{c \in C} (c + p^n{\Z}_p).
\]

A criterion for a subset $C \subset \Z/p^n\Z$ to be $p$-homogeneous is given in \cite{FFS}.
\begin{theorem}[{\cite[Theorem 2.9]{FFS}}] \label{thm5.4}
Let $n$ be a positive integer, and let $C \subseteq {\Z}/p^n{\Z}$ be a multiset. Suppose that
\begin{enumerate}[{\rm(1)}]
    \item $\#C \le p^k$ for some integer $k$ with $1 \le k \le n$;
    \item there exist $k$ integers $1 \le j_1 < j_2 < \dots < j_k \le n$ such that
        \[
        \sum_{c \in C} e^{2\pi i c p^{-j_t}} = 0 \quad \text{for all } 1 \le t \le k.
        \]
\end{enumerate}
Then $\#C = p^k$ and $C$ is $p$-homogeneous. Moreover, the tree ${\mathcal T}_C$ is a ${\mathcal T}_{I}$-form tree with $I = \{j_1-1, j_2-1, \dots, j_k-1\}$.
\end{theorem}

\begin{theorem}[{\cite[Theorem 4.2]{FFS}}]
Let $n$ be a positive integer, and let $C \subseteq {\Z}/p^n{\Z}$. Then $C$ tiles  ${\Z}/p^n{\Z}$ if and only if $C$ is p-homogeneous.
\end{theorem}

\subsection{The structure of tiles in $\Z/p^nq\Z$ and $\Z/p^n\Z\times\Z/p\Z$}
Write  $\Z/p^nq\Z$  as the  product  form $\mathbb{Z}/p^n\mathbb{Z} \times \mathbb{Z}/q\mathbb{Z}$. The geometrical characterization of tiles in $\mathbb{Z}/p^n\mathbb{Z} \times \mathbb{Z}/q\mathbb{Z}$  was obtained in  \cite{FKL}.

 \begin{theorem}[\cite{FKL}]\label{prop5.7}
Let  $A$ be a tile of  $\mathbb{Z}/p^n\Z \times \Z/ q\Z$.   For $0\leq j \leq q-1$, let $A_j=\{x\in \mathbb{Z}/p^n\Z : (x,j)\in A\}$.
\begin{enumerate}[{\rm(1)}]
    \item If   $\#A=p^t$, then  $A_j$ are disjoint and $\cup_{j=0}^{q-1}A_j$ tiles $\Z/p^n\Z$ by translation.
    \item If $\#A=p^tq$, then  all  $A_j$  can tiles  $\mathbb{Z}/p^n\Z $ with a common branched level set.
    \end{enumerate}
\end{theorem}

Consider the finite abelian group $\mathbb{Z}/p^n\mathbb{Z} \times \mathbb{Z}/p\mathbb{Z}$.  Let $ A$ be a tile of $\mathbb{Z}/p^n\mathbb{Z} \times \mathbb{Z}/p\mathbb{Z}$.
It is evident that $\#A$ divide $p^{n+1}$. Hence, assume that  $\#A=p^i$ for some $1 \leq i\leq n$.
%Through out the paper, we assume that $|A|=p^t$ for $2 \leq t \leq n$.
 Define a map $\pi_1$ from the group $ \mathbb{Z}/p^n\mathbb{Z} \times \mathbb{Z}/p\mathbb{Z}$ to the group  $\Z/p^n\mathbb{Z}$ by

\[ \pi_1(a, b)=a, \quad  \hbox{ for } (a, b)\in \mathbb{Z}/p^n\mathbb{Z} \times \mathbb{Z}/p\mathbb{Z}. \]

 Let
 \[\mathcal{Z}_A=\Big\{g\in \mathbb{Z}/p^n\mathbb{Z} \times \mathbb{Z}/p\mathbb{Z}: \widehat {1_ A}(g)=0 \Big\}\] be the set of zeros of the Fourier
transform of the function $1_A$.
The geometrical characterization of tiles in $\mathbb{Z}/p^n\mathbb{Z} \times \mathbb{Z}/p\mathbb{Z}$ was established in \cite{FKL}.

\begin{theorem}[\cite{FKL}]\label{thmfinite}
 Assume  $A$ is  a tile of $\mathbb{Z}/p^n\mathbb{Z} \times \mathbb{Z}/p\mathbb{Z}$ with $\#A=p^t$.
Let
$$
\mathcal{I}_{A}=\big\{i\in \{0,\cdots \ n-1\}: (p^i, 0)\in \mathcal{Z}_ A \big\}.
$$
We distinguish three cases.
\begin{enumerate}[{\rm(1)}]
\item If $\# \mathcal{I}_{A}=t$, then the set $\pi_1(A)$ is a $p$-{\it homogeneous}  in  $\Z/p^n\Z$ with $\#\pi_1( A)=p^t.$
\item If $\# \mathcal{I}_{A}=t-1$ and    $ (p^j, b)\notin \mathcal{Z}_A$  for each  $ j\in  \{0,1,\cdots, n\}\setminus \mathcal{I}_{A}$ and $b\in \{1,\cdots,p-1\}$,  then  the sets
\[ A_i=\{x\in \Z/{p^n}\Z: (x, i) \in A\}\]  are  $p$-homogenous in $\Z/p^n\Z$ with  $\#A_i=p^{t-1}$.
 \item If $\#\mathcal{I}_{A}=t-1$ and   $(p^{j}, b) \in \mathcal{Z}_{A}$ for some  $j\in   \{0,1,\cdots, n\}\setminus \mathcal{I}_{A}$ and $b\in\{1,\cdots,p-1\}$, then
 the set  \[\widetilde{A} = \{ x +b_0 y p^{n-j_0-1}: (x, y) \in   A\} \] is  $p$-homogeneous in $\Z/{p^n}\Z$  with
 $\# \widetilde{A}=p^t$,  where   $j_0$ is the minimal  number in
 $\{0,1,\cdots, n\}\setminus\mathcal{I}_{A}$ such that   $(p^{j_0},b_0)\in \mathcal{Z}_A $ for some  $ b_0\in \{0, \cdots, p-1\}$.
\end{enumerate}
\end{theorem}

\section{Structure of tiles in $\Q_p\times \Z/2\Z$.}
Let $\mu$ be the Haar measure $\Q_p$ such that $\mu(\Z_p)=1$.
Equip   $\Q_p\times \Z/2\Z$ with the Haar measure  $\nu$ with $\nu(\Z_p\times\{0\})=1$.

\subsection{Structure of tiles}
\begin{proof}[Proof of Theorem \ref{thm1.3}]
Let $\left ( \Omega  ,T \right ) $ be a tiling pair in $\Q_p\times \Z/2\Z$.
Write  \[\Omega =\left ( \Omega_{0} \times \left \{ 0 \right \}  \right ) \sqcup \left (  \Omega_{1}\times \left \{ 1 \right \}  \right ),\] and  \[T =\left ( T_{0} \times \left \{ 0 \right \}  \right ) \sqcup \left ( T_{1}\times \left \{ 1 \right \}  \right ).\]

	Since $\left ( \Omega  ,T \right ) $ be a tiling pair  in  $\Q_p\times \Z/2\Z$, by definition, we have
%\begin{align*}
%	A \oplus  T&=\left ( {\Q}_{p}\times \left \{ 0 \right\} \right ) \sqcup \left ( {\Q}_{p}\times \left \{ 1\right \} \right )  \\
%			&=\left ( A_{0} \oplus T_{0} \sqcup  A_{1}\oplus T_{1}  \right ) \times \left \{ 0 \right \}
%			\sqcup  \left (  A_{1}\oplus T_{0} \sqcup  A_{0}\oplus T_{1}  \right ) \times \left \{ 1\right \}.
%		\end{align*}
%		It follows that
%\begin{align*}
%{\Q}_{p} =A_{0} \oplus T_{0} \sqcup A_{0} \oplus T_{1}
%\end{align*}
%and
%\begin{align*}
%{\Q}_{p} =A_{1} \oplus T_{0} \sqcup A_{0} \oplus T_{1},
%\end{align*}
%which is equivalent to
	\begin{equation}\label{equ61}
		1_{ \Omega_{0} } \ast \mu_{T_{0} }+1_{ \Omega_{1}}\ast \mu_{T_{1} }=1,   \quad \mu-a.e.
	\end{equation}
	and
	\begin{equation}\label{equ62}
		1_{ \Omega_{1} } \ast \mu_{T_{0} }+1_{ \Omega_{0}}\ast \mu_{T_{1} }=1, \quad  \mu-a.e.
	\end{equation}
	where $\mu_T=\sum_{t\in T} \delta_t $ is a discrete measure in $\Q_p$.
	
	 Next, we will transform the tiling problem on $\mathbb{Q}_p \times \mathbb{Z}/2\mathbb{Z}$  by  a set  into tiling problem on $\mathbb{Q}_p$ by a function. Adding both sides of Equation~\eqref{equ61} and Equation~\eqref{equ62} separately,  we get
	\begin{equation}\label{equ63}
		\left ( 1_{ \Omega_{0} } +1_{ \Omega_{1}} \right ) \ast \left ( \mu_{T_{0} } +\mu_{T_{1} } \right ) =2.
	\end{equation}
	Subtracting Equation \eqref{equ62}  from Equation \eqref{equ61}, we get
\begin{equation}\label{equ64}
	\left ( 1_{ \Omega_{0} } -1_{ \Omega_{1}} \right ) \ast \left ( \mu_{T_{0} } -\mu_{T_{1} } \right ) =0.
\end{equation}
	Hence, by Theorem  \ref{thm-tilingperiodic}, we deduce from \eqref{equ63} that the function  $ f= 1_{ \Omega_{0} } +1_{ \Omega_{1}}$ is uniformly locally constancy. Then, we have $ \Omega_{0} \cap  \Omega_{1} $ and $\left (  \Omega_{0} \setminus   \Omega_{1}  \right ) \cup \left (  \Omega_{1} \setminus   \Omega_{0}  \right ) $ are almost compact open.

Now, we distinguish two cases: $\left ( 1 \right )$ If $\mu_{T_{0}}  -\mu_{T_{1}}=0 $;  $\left ( 2 \right )$ If $\mu_{T_{0}}  -\mu_{T_{1}}\neq 0 $.

$\left ( 1 \right )$ If $\mu_{T_{0}}  -\mu_{T_{1}}=0 $, then we deduce from (\ref{equ63}) that
$$\left ( 1_{ \Omega_{0} } +1_{ \Omega_{1}} \right ) \ast \mu_{T_{0}}=1.$$
Hence $\left (  \Omega_{0}\cup  \Omega_{1} ,T_{0} \right ) $ is a tiling pair in  ${\Q}_{p} $ and  $ \Omega_{0}\cup  \Omega_{1}$ is almost compact open.

$\left ( 2 \right )$ If $\mu_{T_{0}}  -\mu_{T_{1}}\neq 0 $, then by Theorem  \ref{thm-tilingperiodic}, we deduce from \eqref{equ64} that the function  $ f= 1_{ \Omega_{0} } -1_{ \Omega_{1}}$ is uniformly locally constancy, which implies that  $ \Omega_{0} \setminus   \Omega_{1}$ and $ \Omega_{1} \setminus   \Omega_{0}$ are almost compact open. Hence, both $ \Omega_0$ and $ \Omega_1$ are compact open. Without loss of generality, assume that   $ \Omega_0,  \Omega_1$ are compact open set in $\Z_p$. Hence, there exist a positive integer $n$ and $C_0,C_1 \in \Z/p^n \Z$ such that
\[ \Omega_0=\bigsqcup_{c \in C_0} c + p^n \mathbb{Z}_p,  \quad   \Omega_1=\bigsqcup_{c \in C_1} c + p^n \mathbb{Z}_p,\]
up to measure zero sets.
  Note that $ \Omega_0\times\{0\} \cup  \Omega_1\times \{1\}$ tiles $\Q_p\times\Z/2\Z$ by translation  if and only if   $C_0\times\{0\} \cup C_1\times \{1\}$  tiles  $\Z/p^n \Z \times\Z/2\Z$ by translation. 
  
  When  $p\geq 3$, by Theorem \ref{prop5.7},  either $C_0\cap C_1=\emptyset$ and $C_0\cup C_1$ tiles $\Z/p^n\Z$ by translation or  $C_0, C_1$  can tile $\Z/p^n\Z$ by translation with a common translation set.
  
  When $p=2$, by Theorem \ref{thmfinite}, we have three cases: (1) $C_0\cap C_1=\emptyset$ and $C_0\cup C_1$ tiles $\Z/2^n\Z$ by translation, (2) $C_0, C_1$  can tile $\Z/2^n\Z$ by translation with a common translation set, (3) there exist $j_0\in \{0,\cdots, n-1\}$ and such that
$C_0$ and  $\widetilde{C}_1 = \{ x +  2^{n-j_0-1}: x  \in C_1\}$ are disjoint and  $C_0\cup \widetilde{C}_1$ tiles  $\Z/2^n\Z$ by translation.

The combination of the two cases led to our final conclusion.
\end{proof}

\subsection{Tiles are spectral in $\Q_p\times \Z/2\Z$}
It is proved  in \cite{FFS,FFLS}that Fuglede conjecture holds in  $\Q_p$, spectral sets are compact open and are characterized by their  $p$-homogeneity.
\begin{theorem}[\cite{FFS}]\label{thm6}
	Let $\Omega $ be a compact open set in $\Q_{p}$. The following statements are equivalent:
	\begin{enumerate}[{\rm(1)}]
		\item $\Omega $ is a spectral set;
		\item $\Omega $ tiles $\Q_{p}$ by translation;
		\item $\Omega $ is $p$-homogeneous.
	\end{enumerate}
\end{theorem}
Let $E$ be a subset of $\Q_{p}$. Define
$$I_{E} =\left \{ i\in \Z:\:\exists \: x,y\in E \:  such \:that \: v_{p}\left ( x-y \right )=i   \right \} $$
as the set of admissible $p$-orders corresponding to the set $E$.
The structure of spectra for spectral sets is also characterized.
\begin{theorem}[\cite{FFS} Theorem 5.1]\label{thm5}
	Let $\Omega \subset \Q_{p} $ be a $p$-homogeneous compact open set with the admissible $p$-order set $I_{\Omega } $.
	\begin{enumerate}[{\rm(1)}]
		\item The set $\Lambda $ is a spectrum of $\Omega$ if and only if it is $p$-homogeneous discrete set with admissible $p$-order set
		$I_{\Lambda }=-\left ( I_{\Omega }+1 \right ) $.
		\item The set $T$ is a tiling complement of $\Omega$ if and only if it is a $p$-homogeneous discrete set with admissible $p$-order set $I_{T} =\Z \setminus I_{\Omega }$.
	\end{enumerate}
\end{theorem}

Let $\Omega\subset \Q_p\times \Z/2\Z$ be a tile.   Write  $\Omega$ as  \[\Omega =\left ( \Omega_{0} \times \left \{ 0 \right \}  \right ) \sqcup \left (  \Omega_{1}\times \left \{ 1 \right \}  \right ).\] 

When $p>2$, by Theorem \ref{thm1.3},   there are two mutually exclusive cases:
{\rm (1)} $\mu(\Omega_0\cap\Omega_1)=0$ and  $ \Omega_0\cup \Omega_1$ can tile $\Q_p$ by translation,
{\rm (2)} $\Omega_0$ and $\Omega_1$  can tiles $\Q_p$ with  a common translation set  $T_0\subset \Q_p$.
For the first case, $\Omega_0\cup \Omega_1 $ tiles $\Q_p$ by translation. Assume that $(\Omega_0\cup \Omega_1,\Lambda)$ is a spectral pair.  Therefore, $(\Omega, \Lambda \times \left \{ 0 \right \})$ is a spectral pair in $\Q_p\times \Z/2\Z$.
For the second case, the set $\Omega_0$ tiles $\Q_p$ by translation. It follows that $\Omega_0$  and $\Omega_1$  are spectral sets in $\Q_p$ with a common spectrum $\Lambda$.  Therefore, $(\Omega,\left ( \Lambda \times \left \{ 0 \right \}  \right ) \sqcup \left (  \Lambda \times \left \{ 1 \right \}  \right ))$ is also a spectral pair.

When $p=2$, by Theorem \ref{thm1.3},   there are three cases. For the first two cases, it is similar as $p>2$. For the third case, by Theorem \ref{thm1.3}, $\Omega_0$ and $\Omega_1$ are compact open up to  measure zero sets. Without loss of generality, assume that   $ \Omega_0,  \Omega_1$ are compact open set in $\Z_2$. Hence, there exist a positive integer $n$ and $C_0,C_1 \in \Z/2^n \Z$ such that
\[ \Omega_0=\bigsqcup_{c \in C_0} c + 2^n \mathbb{Z}_2,  \quad   \Omega_1=\bigsqcup_{c \in C_1} c + 2^n \mathbb{Z}_2.\]
Hence, $C=C_{0}\times\{0\} \cup C_{0}\times\{1\}$ tiles $\Z/2^n \Z \times \Z/2\Z$ by translation. Let
$$
\mathcal{I}_{C}=\big\{i\in \{0,\cdots \ n-1\}: (2^i, 0)\in \mathcal{Z}_ C \big\},
$$
and let $j_0\in \{0,1,\cdots, n-1\}\setminus  \mathcal{I}_{C}$ such that   $(2^{j_0},1)$ in $\mathcal{Z}_C$ and $(2^i, 1)\notin \mathcal{Z}_C$ for $i\in  \mathcal{I}_{C}$ with $i<j_0$.
 
Define 
$$
\Lambda=\Big\{(\sum \limits_{i\in \mathcal{I}_{C}}s_ip^{i-n}, 0)+s_{j_0}(p^{j_0-n},1): s_i, s_j \in \{0, 1, \cdots, p-1\} \Big\}+(\mathbb{L}_{n}\times\{0\}), 
$$
where \[\mathbb{L}_n=\{\frac{k}{p^{l}}: l> n, 1\leq k\leq p^{l} \hbox{ and } \gcd(k,p)=1\}.\]
Hence, $(\Omega, \Lambda)$ forms a spectral pair of $\Q_2\times \Z/2\Z$. 

%By Theorem \ref{thm6}, we have $A_{0}\cup A_{1}$ is $p$-homogeneous and almost compact open. Then according to Theorem \ref{thm5}, we can easily find the spectrum $\Lambda $ of $A_{0}\cup A_{1}$. Hence $A_0$ and $A_1$ are compact open.
%Let $\mu $ be a probability Borel measure on $\Q_{p}$. We say that $\mu $ is a spectral measure if there exists a set $\Lambda \subset \hat{\Q} _{p} $ such that $\left \{ \chi _{\lambda }  \right \} _{\lambda \in \Lambda } $ is an orthonormal basis of $L^{2} \left ( \mu  \right ) $. Then $\Lambda$ is called a spectrum of $\mu $ and we call $\left ( \mu ,\Lambda  \right ) $ a spectral pair. Assume that $\Omega $ is a set in $\Q_{p}$ of positive and finite Haar measure. That $\Omega $ is a spectral set means the restricted measure $\frac{1}{\mathfrak{m}\left ( \Omega  \right )  }\mathfrak{m}\mid _{\Omega }  $ is a spectral measure. In this case, instead of saying $\left ( \frac{1}{\mathfrak{m}\left ( \Omega  \right )  }\mathfrak{m}\mid _{\Omega }  ,\Lambda  \right ) $ is a spectral pair, we say
%\bibliographystyle{plain}
%\bibliography{cankaowx.bib}

%\bibliographystyle{plain}
%\bibliography{ref}

\end{document}